\documentclass[11pt]{amsart}
\usepackage{amssymb,amscd,verbatim,graphicx}
\usepackage[T1]{fontenc}
\usepackage{mathrsfs}
\usepackage{xy}
\xyoption{all}
\usepackage[colorlinks,linkcolor=blue,citecolor=blue,urlcolor=red]{hyperref}

\swapnumbers
\newtheorem{lemma}{Lemma}[section]
\newtheorem{prop}[lemma]{Proposition}
\newtheorem{thm}[lemma]{Theorem}
\newtheorem{cor}[lemma]{Corollary}

\theoremstyle{definition}
\newtheorem{defn}[lemma]{Definition}
\theoremstyle{remark}
\newtheorem{rem}[lemma]{Remark}
\newtheorem{ex}[lemma]{Example}

\newcommand{\Qvsp}{\mathbb{Q}\mathrm{-vsp.}}

\newcommand{\gr}{\mathrm{gr}}
\newcommand{\Q}{\mathbb{Q}}
\newcommand{\Z}{\mathbb{Z}}
\newcommand{\C}{\mathbb{C}}

\newcommand{\unit}{\mathbf{1}}
\newcommand{\Ah}{\mathcal{A}}
\newcommand{\Bh}{\mathcal{B}}
\newcommand{\Ch}{\mathcal{C}}
\newcommand{\Th}{\mathcal{T}}
\newcommand{\pathcat}{\mathcal{P}}
\newcommand{\isom}{\cong}
\newcommand{\tensor}{\otimes}
\newcommand{\Hom}{\mathrm{Hom}}

\newcommand{\Ker}{\mathrm{Ker}}

\newcommand{\id}{\mathrm{id}}
\newcommand{\Spec}{\mathrm{Spec}}

\newcommand{\catF}[1]{\mathrm{Ab}(#1)}
\newcommand{\catN}[1]{\mathrm{\mathcal{A}}(#1)}
\newcommand{\neeman}[1]{\mathrm{Ab}^\Delta(#1)}
\newcommand{\catNsub}[2]{\mathrm{\mathcal{A}}_{#1}(#2)}
\renewcommand{\mod}{\mathrm{-mod}}

\newcommand{\Proj}{\mathrm{-proj}}
\newcommand{\op}{\mathrm{op}}

\newcommand{\DMgm}{\mathrm{DM}_{\mathrm{gm}}}
\newcommand{\MMN}{\mathsf{MM}}
\newcommand{\ECM}{\mathsf{ECM}}
\newcommand{\eff}{\mathrm{eff}}
\newcommand{\NM}{\mathsf{NM}}
\newcommand{\Nori}{\mathrm{Nori}}
\newcommand{\good}{\mathrm{good}}
\newcommand{\sign}{\mathrm{sgn}}
\newcommand{\Sch}{\mathrm{Sch}}

\title{Tensor structure for Nori motives}
\author{Luca Barbieri-Viale}\thanks{The first author acknowledges the support of the {\it Ministero dell'Istruzione,  dell'Uni\-ver\-sit\`a e della Ricerca} (MIUR)  through the Research Project  (PRIN 2010-11) ``Arithmetic Algebraic Geometry and Number Theory'' and the Freiburg Institute of Advanced Study.}
\address{Dipartimento di Matematica ``F. Enriques'', Universit{\`a}
degli Studi di Milano\\ via C. Saldini, 50\\ Milano I-20133\\ Italy}
\email{luca.barbieri-viale@unimi.it}
\author{Annette Huber}\thanks{The second author thankfully acknowledges support by Freiburg Institute of Advanced Study during the preparation of this note.}
\address{Math. Institut, Universit\"at Freiburg, Ernst-Zermelo-Str. 1, 79104 Freiburg, Germany}
\email{annette.huber@math.uni-freiburg.de}
\author{Mike Prest}\thanks{
The third author gratefully acknowledges the support of the Freiburg Institute of Advanced Study.}
\address{Department of Mathematics, Alan Turing Building, University of Manchester, Manchester M13 9PL, UK}
\email{mprest@manchester.ac.uk}

\begin{document}
\begin{abstract}
We construct a tensor product on Freyd's universal abelian category $\catF{C}$ attached to an additive tensor category or a $\tensor$-quiver and establish a universal property. This is used to give an alternative construction for the tensor product on Nori motives.
\end{abstract}

\maketitle

\tableofcontents

\section*{Introduction}
In the late 1990's Nori made a spectacular proposal for an unconditional definition of an abelian category of motives and a motivic Galois group over a field of characteristic zero.  It has two main inputs:
\begin{enumerate}
\item The existence of a universal abelian category attached to a fixed representation of a quiver.
\item His Basic Lemma (known earlier to Beilinson and Vilonen) which shows the existence of an algebraically defined
``skeletal filtration'' on an affine algebraic variety.
\end{enumerate}
The first part is enough to give the definition of the category. The second is needed in order to establish the tensor structure. In a third step, we pass from effective motives to all motives and check rigidity.

The motivic Galois group is its Tannaka dual.
However, all steps are intrinsically linked together. The proof of the existence of the abelian category is done by constructing a suitable coalgebra. The tensor product is defined by turning this coalgebra into a bialgebra. After localisation, it is shown to be even a Hopf algebra - the Hopf algebra of the motivic Galois group. Indeed, the proof given in full detail in \cite{HMS} gives as a byproduct a full proof of Tannaka duality.

Meanwhile there have been a couple of alternative approaches to the first step of the above program, see \cite{BVCL}, \cite{BV}, \cite{BVP} and \cite{Ivorra}. They are more general and arguably simpler.
However, these references did not address tensor products.

In this paper we explain how the approach of \cite{BVP} can be used to handle tensor categories and tensor functors.  We show that if $(C,\otimes)$ is an additive tensor category then Freyd's universal abelian category $\catF{C}$ carries an induced right-exact tensor structure which is also universal in a certain sense (the exact statement is Proposition~\ref{prop:right_exact}).

Given a module $M$ on $C$ (i.e., an additive functor into an abelian tensor category), this induces, under additional technical assumptions, a tensor structure on the universal
abelian category $\catN{M}$ for the module $M$. It is again universal, see Proposition~\ref{tensor:quot}. The results can also be reformulated in terms of representations of quivers, see Section~\ref{sect:quiver}, in particular Theorem~\ref{thm:Norigd}, bringing it even closer to the shape of Nori's original results. Our results are a lot more general in allowing modules with values in quite general abelian categories. We get back Nori's case as the case of the representation of a quiver in the category of modules over a Dedekind domain or even a noetherian ring of homological dimension at most $2$.

We also show how to apply our results to Nori motives. This can be done by using his original quiver of good pairs. Alternatively, we start with the more canonical tensor category of geometric motives in the sense of Voevodsky.
However, the functor $H^0_B$ used in the definition of
Nori motives is \emph{not} a tensor functor, in contrast with the
graded functor $H^*_B$. It remains to check that the K\"unneth components
are motivic. It is this step of the construction that relies on Nori's Basic Lemma. We give an abstract criterion in Section~\ref{sec:kuenneth} below. It is applied to Nori motives in Section~\ref{sec:motives}:
we obtain a Tannakian category and define the motivic Galois group as its
Tannaka dual. We find this more natural than defining the category as representations of the motivic Galois group. 

We feel that the nature of the argument and the role of the Basic Lemma become a lot clearer in this new description. However, its main advantage is the great generality in choosing the target category $\Ah$. E.g. we can easily define Nori motives over a base by using the Betti-realisation of triangulated motives into constructible sheaves.  In a follow-up, see \cite{BVP2}, two of us take a more axiomatic approach, using many-sorted languages such that (co)homology theories are models of certain regular theories in that language.  All this applies to several different geometric situations.

\subsection*{Acknowledgements}
Thanks to Ralph Kaufmann for bringing our attention to the references \cite{Bun} and \cite{Day}.
We thank Paul Balmer for pointing us to \cite{BKS} and subsequent discussions of the case of homological functors, see Section~\ref{sec:kuenneth}. We thank the referee for her/his suggestions that considerably improved the exposition.

\subsection*{Notation}
By a tensor category $(C, \otimes)$ we mean a category $C$ provided with a functor $\otimes: C\times C\to C$ satisfying an associativity constraint and with $\unit$ a unit object; in addition, also a commutativity constraint can be required, e.g. see \cite[\S 1]{DMT}.  This is often called a non-strict tensor category.
By an additive (resp.\/ abelian) tensor category we mean a tensor category $(C, \otimes)$ such that $C$ is additive (resp.\/ abelian) and $\tensor$ is a bi-additive functor, see \cite[Def. 1.15]{DMT}. Tensor functors are not assumed strict.
Tensor functors between additive tensor categories are assumed to be additive. We denote by $\Qvsp$ the tensor category of $\Q$-vector spaces.

If $\Ah$ is an abelian category, we denote by $\gr\Ah$ the associated category of $\Z$-graded objects. If, in addition, $(\Ah , \otimes)$ carries a tensor structure, we equip $(\gr\Ah , \otimes )$ with the induced tensor structure.
If the  tensor product is commutative, we choose the commutativity constraint on $\gr\Ah$ such that the product becomes graded anti-commutative.

For an additive category $C$ we shall consider the additive functors from $C$ to the category Ab of abelian groups as (left) $C$-modules. We shall denote by $C\mod$ the category of finitely presented $C$-modules, see e.g.~\cite[Chaps.~2, 3]{P}.

\section{Universal abelian tensor categories}
\label{sec:universal_tensor}

Let $C$ be an additive category. We denote by $\catF{C}$ the universal abelian category on $C$, see \cite{Frey}, also \cite[Chap.~4]{P}. We may refer to it as Freyd's abelian category. It comes with a canonical fully faithful functor $C\hookrightarrow\catF{C}$. Recall that this functor is universal with respect to additive functors into abelian categories, e.g. see \cite[Thm 1.1]{BVP}.

Thus, for $M:C\to\Ah$ an additive functor into some abelian category $\Ah$, we obtain an induced exact functor $\widetilde{M}:\catF{C}\to\Ah$, unique to natural equivalence.

We denote by $\catN{M}$ the quotient of $\catF{C}$ by the Serre subcategory which is the kernel of $\widetilde{M}$; we also denote by $\widetilde M: \catN{M}\to \Ah$ the induced faithful exact functor.

$$
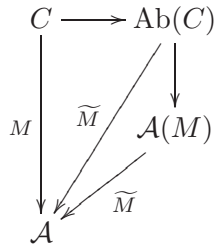
\xymatrix{C \ar[r] \ar[dd]_M & \catF{C} \ar[ddl]_{\widetilde{M}} \ar[d] \\
 & \catN{M} \ar[dl]^{\widetilde{M}} \\
\Ah}$$

We shall refer to $\catN{M}$ as the {\it universal abelian category defined by $M$}, according with \cite[\S 1.1]{BVP}. In fact, this abelian category $\catN{M}$ is universal for (i.e.~initial among) all abelian categories together with
a faithful exact functor into $\Ah$ which extends  $M$.  Note that, in the case where $\Ah$ is the category of finitely generated modules over a commutative  noetherian ring $R$, this recovers Nori's abelian category (see \cite[Chap.~7]{HMS} and compare with \cite[\S 1.2]{BVP}).
For later use, we introduce:
\begin{defn}\label{defn:subcat}
Let $C$ be additive and let $C\to \catF{C}$ be Freyd's abelian category. We denote by $\catF{C}^\flat$ the smallest full subcategory containing the objects in the image of $C$ and closed under kernels.
\end{defn}
\begin{rem}\label{rem:prime}The universal abelian category $\catF{C}$ can be constructed explicitly as the category $(C\mod)\mod$ (see, e.g., \cite[4.3]{P}). In this construction,
$\catF{C}^\flat$ is, because $C\mod$ has cokernels and every object of $C\mod$ is the cokernel of a morphism between representables, precisely the image of
$C\mod$ under the (contravariant) Yoneda embedding into $\catF{C}$.  All objects of $\catF{C}^\flat$ are representable functors hence, by the Yoneda lemma, projective.  The above definition is independent of this description.

\end{rem}
Let $(C, \tensor )$ be an additive tensor category, see \cite[\S 1]{DMT}. Consider an (additive) tensor functor $M: (C , \tensor)\to (\Ah, \tensor )$ where $(\Ah , \tensor )$ is an abelian tensor category. We want to equip the above universal abelian category $\catN{M}$ with a natural tensor structure $(\catN{M},\tensor )$ such that
$\widetilde{M}: (\catN{M},\tensor )\to (\Ah , \tensor)$ is turned into a tensor functor. We proceed in several steps.

\subsection*{Multilinear functors}
By definition, $\catF{C}$ has a universal propery with respect to additive functors. In fact, this extends to bi-additive and even multi-additive functors, even though we lose some properties.

We first recall a well-known property of injective resolutions.
\begin{lemma}\label{lem:resol}
Let $\Ah$ be an abelian category, $f:X\to Y$ a morphism in $\Ah$. Assume that
\[ 0\to X\to I_0\to I_1, \quad 0\to Y\to J_0\to J_1\]
are exact and that all $I_k$, $J_k$ are injective. Then
there are lifts $f_0:I_0\to J_0$, $f_1:I_1\to J_1$ making the diagram
\[\begin{xy}\xymatrix{
0\ar[r]&X\ar[r]\ar[d]_f&I_0\ar[r]\ar[d]_{f_0}\ar[r]^d &I_1\ar[d]_{f_1}\\
0\ar[r]&Y\ar[r]&J_0\ar[r]&J_1
}\end{xy}\]
commute. Moreover, if $(g_0,g_1)$ is a second lift, then there is
$h:I_1\to J_0$ such that
\[ f_0-g_0=h\circ d.\]
\end{lemma}
\begin{proof}
Our complexes are the starting bits of injective resolutions
and $h$ is the beginning of a chain homotopy. The assertion is
usually proved as the first step of the proof of existence of a lift of $f$ to
an injective resolution and that the lift is unique up to chain homotopy,
see for example \cite[Theorem~6.1]{maclane}.
\end{proof}
As usual, we also have the dual statement for left resolutions by projectives.

\begin{prop}\label{prop:mult}
Let $C_1,\dots,C_n$ be additive categories, $\Ah$ an abelian category.
\begin{enumerate}
\item
 Let
$F:C_1\times\dots \times C_n\to \Ah$ be a multilinear functor, i.e., additive in each argument. Then $F$ extends to a multilinear functor
\[ \widetilde{F}:\catF{C_1}\times \catF{C_2}\times\dots\times\catF{C_n}\to \Ah\]
which is right-exact in each argument. Fix $j$ and for $i\neq j$ choose $X_i\in\catF{C_i}^\flat$ (see Definition~\ref{defn:subcat}). Then $\widetilde{F}(X_1,\dots,-,\dots,X_n)$ is exact as a functor on $\catF{C_j}$.
\item The functor $F$ is uniquely determined up to unique isomorphism of functors by these properties.
\item Let $\alpha:F_1\to F_2$ be a transformation of multilinear functors
$C_1\times\dots\times C_n\to\Ah$ and $\tilde{F}_1$ and $\tilde{F}_2$ their extensions to $\catF{C_1}\times\dots\times\catF{C_n}$. Then there is a transformation of functors $\tilde{\alpha}:\tilde{F}_1\to\tilde{F}_2$ extending $\alpha$. It is unique.
\end{enumerate}
\end{prop}
\begin{proof}

Recall that $\catF{C_i}=(C_i\mod)\mod$ and that the universal functor factors
\[ C_i\to (C_i\mod)^\op\to (C_i\mod)\mod\]
where both steps are given by the Yoneda embedding. As pointed out in Remark~\ref{rem:prime} the subcategory $\catF{C_i}^\flat$ agrees with the image of $(C_i\mod)^\op$.

All statements are shown in two steps. In the first we extend to a functor
\[ F':(C_1\mod)^\op\times (C_2\mod)^\op \times \dots \times (C_n\mod)^\op \to \Ah\]
which will be multilinear and left exact in each argument.
In the second step, which is actually dual to the first, we extend $F'$ to
$\tilde{F}$.

We first show uniqueness. This will make clear why the formula that
we use in the construction is correct.
Let $E$ be any extension of $F$ to $\catF{C_i}$ with the
exactness property of (1).
Let $X_i\in (C_i\mod)^\op$.
We argue by descending induction on the number of $X_i$ which are in the image of
$C_i$, i.e., of the form $(A_i,-)^\op$. If they all are, then
$E(X_1,\dots,X_n)=F(A_i,\dots,A_n)$ by assumption. Assume that
$E$ is uniquely determined if at least $m$ of the $X_i$ are corepresentable.
After reordering we have to consider the tuple $(X_1,\dots,X_n)$
with $X_i=(A_i,-)^\op$ for $i<m$. By definition, there is an
injective corepresentation
\[ 0\to X_m\to (A_m,-)^\op\to (B_m,-)^\op.\]
By (1), the functor $E(X_1,\dots,X_{m-1},-,X_{m+1},\dots,X_n)$
is exact. Hence we have an exact sequence
\begin{multline*}
 0\to E(X_1,\dots,X_n)\to
E(X_1,\dots,X_{m-1},(A_m,-)^\op,X_{m+1},\dots,X_n)\\
\to E(X_1,\dots,X_{m-1},(B_m,-)^\op,X_{m+1},\dots,X_n).
\end{multline*}
By induction the two terms on the right are uniquely determined up to unique isomorphism.
As a kernel, $E(X_1,\dots,X_n)$ is again uniquely determined up to unique isomorphism.
By induction, this shows uniqueness if all arguments are in
$C_i\mod$.

The dual argument for a right exact $E$ and representations gives uniqueness for arguments in $\catF{C_i}$.

We turn to the construction of $F'$. Let $X_i\in (C_i\mod)^\op$.
By definition, these objects have an injective copresentation
\[ 0\to X_i\to (A_i,-)^\op\to (B_i,-)^\op.\]
We choose such a presentation for each object $X_i\in (C_i\mod)^\op$.
The uniqueness proof suggests $F'(X_1,\dots,X_n)\subset F(A_1,\dots,A_n)$
as an iteration of kernels. The same object is given by the formula
\[ F'(X_1,\dots,X_n):=\Ker\left(F(A_1,\dots,A_n)\to \bigoplus_{m=1}^n F(A_1,\dots,A_{m-1},B_m,A_{m+1},\dots,A_n)\right).\]
In other words, applying $F$ to the $n$-tuple of complexes $(A_i,-)^\op\to (B_i,-)^\op$ we obtain an $n$-fold complex. The above is $H^0$ of its total complex.

Let $1\leq i\leq n$. For a morphism $X_i\to Y_i$ in $(C_i\mod)^\op$ we choose a lift to the copresentations as in Lemma~\ref{lem:resol}. This induces a morphism
$F'(X_1,\dots,X_n)\to F'(X_1,\dots,Y_i,\dots,X_n)$. It is independent of
the lift because any two such differ by $h$ as in Lemma~\ref{lem:resol}. This makes $F'$ a functor in each variable.

The dual argument via projective presentations gives a right-exact extension to
$\catF{C_i}$.

A diagram chase shows that the functor $\tilde{F}$ has the exactness property claimed in (1) because
$F'$ is left exact, $\tilde{F}$ right exact and every object $Y$ of
$\catF{C_i}$ has a projective resolution of the form
\[ 0\leftarrow Y\leftarrow P^0\leftarrow P^1\leftarrow P^2\leftarrow 0\]
with $P^i\in (C_i\mod)^\op$.

Now let $\alpha:F_1\to F_2$ be a transformation of functors. Going through
the above construction, we get induced $\alpha':F'_1\to F'_2$ and then
$\tilde{\alpha}:\tilde{F}_1\to \tilde{F}_2$. The uniqueness argument for
the functors also gives the uniqueness of the transformation.
\end{proof}

\begin{rem}Unexpectedly
the extension $\tilde{F}$ \emph{fails} to be exact in each argument.
For a counterexample, see Example~\ref{ex:Zvier} below.
\end{rem}

\begin{rem} In a general enriched-category setting, lifting monoidal structure to functor categories can be found in \cite{Bun} and \cite{Day}.
\end{rem}

This applies in particular to additive tensor categories.

\begin{defn}\label{FT}
Let $(C, \otimes)$ be an additive tensor category. We extend the functor $\tensor:C\times C\to C$ defining
\[ \tensor:\catF{C}\times\catF{C}\to \catF{C}\]
as the extension of $C\times C\to C\hookrightarrow \catF{C}$ of Proposition~\ref{prop:mult}.
\end{defn}

\begin{prop}\label{prop:tensor}
Let $(\catF{C}, \otimes )$ be Freyd's category together with the functor in Definition~\ref{FT}. Then
\begin{enumerate}
\item $(\catF{C}, \otimes )$ is an abelian tensor category.
\item The tensor product is right-exact by construction. The objects in $\catF{C}^\flat$ are flat, i.e., acyclic with respect to $\tensor$.
\item If the tensor structure on $C$ is commutative
then so is the tensor structure on $\catF{C}$.
\end{enumerate}
\end{prop}
\begin{proof}
Right-exactness and acyclicity are special cases of Proposition~\ref{prop:mult}.
Let $\unit$ be the unit object of $C$. By definition it comes with a transformation of functors $u:\unit\tensor-\to\id$ on $C$. Let $[\unit]$ be its image in
$\catF{C}$. In explicit formulas this means $[\unit]=((\unit,-),-)$.
Then $[\unit]$
with the induced transformation is the unit
of $\catF{C}$.
The equivalences used to express the associativity constraint on $C^3$ (see \cite[\S 1]{DMT}) induce equivalences
on $\catF{C}^3$. In detail: Let
$F_1:C^3\xrightarrow{\tensor\circ (\id,\tensor)}C$ and
$F_2:C^3\xrightarrow{\tensor\circ(\tensor,\id)}C$. The associativity constraint is a functorial isomorphism $\alpha:F_1\to F_2$. By abuse of notation we use the same notation for their composition with the inclusion
$C\to\catF{C}$. Note that $\alpha$ is still a functorial isomorphism.
The functor
$\catF{C}^3\to \catF{C}$ given by $(X,Y,Z)\mapsto X\tensor(Y\tensor Z)$ is right-exact in each argument and exact as a functor in one variable if the other entries are flat. By the uniqueness property of Proposition~\ref{prop:mult} it agrees with $\tilde{F}_1$. The same argument also applies to
$F_2$. Again by Proposition~\ref{prop:mult}, the transformation extends to a
transformation $\tilde{\alpha}$. This is our associativity constraint. We need to check that a certain diagram of functors on $\catF{C}^4$ involving
$\tensor$ and $\tilde{\alpha}$ commutes. This holds by the uniqueness part
of Proposition~\ref{prop:mult} applied to functors $C^4\to C$.

We argue similarly for the commutativity constraint if there is one on $C$.
\end{proof}
\begin{defn}\label{def:flat-subcat}
For an abelian tensor category, with a right exact tensor product, a \emph{$\flat$-subcategory} is a full additive subcategory of flat objects (i.e., acyclic with respect to the tensor product) which is closed under kernels. If $(\Ah, \tensor)$ is such an abelian tensor category we shall denote by $\Ah^\flat \subseteq \Ah$ some $\flat$-subcategory.
\end{defn}
As a consequence of Proposition~\ref{prop:tensor} we have that $\catF{C}^\flat \subset \catF{C}$ as in Definition~\ref{defn:subcat} is a
$\flat$-subcategory.
\begin{prop}[Universal property]\label{prop:right_exact}
Let $C$ be an additive tensor category. Let $\Ah$ be an abelian tensor category with a right exact tensor product. Let $M: (C , \tensor)\to (\Ah, \tensor )$ be a tensor functor. In addition, assume that $M$ factors via $\Ah^\flat\subseteq\Ah$ a $\flat$-subcategory  (see Definition \ref{def:flat-subcat}).
Then $\widetilde{M}:(\catF{C}, \tensor)\to (\Ah, \tensor )$ is a tensor functor.
The triple $(\catF{C},\catF{C}^\flat,\tensor)$ is universal with this property, in particular unique.
\end{prop}

\begin{proof}Let $M: (C , \tensor)\to (\Ah, \tensor )$ be a tensor functor. We have to compare
\[ \catF{C}\times\catF{C}\to\catF{C}\to\Ah\]
and
\[ \catF{C}\times\catF{C}\to \Ah\times\Ah\to\Ah.\]
Both are right-exact in each argument (this is where right-exactness of the tensor product on $\Ah$ is used) and agree on $C\times C$.

As in the proof of Proposition~\ref{prop:mult}, we extend $M$ in two steps: first to $(C\mod)^\op$, then to $\catF{C}=(C\mod)\mod$.
The second step is unproblematic as it only uses the right-exactness.
In the first step, we need to check the action on (certain)
kernels. Let $X_1,X_2\in (C\mod)^\op$ with
resolutions
\[ 0\to X_i\to (A_i,-)^\op\rightarrow (B_i,-)^\op.\]
By definition
\[ 0\to M'(X_i)\to M(A_i)\to M(B_i)\]
is exact.
By assumption $M(A_i), M(B_i)$ and hence also $M'(X_i)$ are in $\Ah^\flat$. In particular,
consider the diagram
\[\begin{xy}\xymatrix{
 &0\ar[d]&0\ar[d]&0\ar[d]\\
0\ar[r]&M'(X_1)\tensor M'(X_2)\ar[r]\ar[d]&M'(X_1)\tensor M(A_2)\ar[r]\ar[d]&M'(X_1)\tensor M(B_2)\ar[d]\\
0\ar[r]&M(A_1)\tensor M'(X_2)\ar[r]\ar[d]&M(A_1)\tensor M(A_2)\ar[r]\ar[d]&M(A_1)\tensor M(B_2)\ar[d]\\
0\ar[r]&M(B_1)\tensor M'(X_2)\ar[r]&M(B_1)\tensor M(A_2)\ar[r]&M(B_1)\tensor M(B_2)
}\end{xy}\]
All rows and columns are exact because they arise by tensoring an exact sequence with a flat object. This implies
\begin{multline*} M'(X_1)\tensor M'(X_2)=\Ker\big( M(A_1)\tensor M(A_2)\to (M(A_1)\tensor M(B_2))\oplus (M(B_1)\tensor M(A_2))\big)\\
=M'(X_1\tensor X_2).
\end{multline*}

The triple $(\catF{C},\catF{C}^\flat,\tensor)$ satisfies itself the assumptions of the universal property, hence it is universal and as such unique.
\end{proof}

\begin{rem}\label{rem:univ}
There are a number of interesting cases where the assumptions of Proposition \ref{prop:right_exact} and Definition \ref{def:flat-subcat} are satisfied. However, they are not as general as one could hope for.
\begin{enumerate}
\item If $\tensor$ is exact on $\Ah$, then $\Ah^\flat=\Ah$ clearly satisfies the assumptions.
\item If $C_1\to C_2$ is a $\tensor$-functor between additive tensor categories, by composition, we may consider $M : C_1 \to \Ah^\flat =\catF{C_2}^\flat\subset  \Ah = \catF{C_2}$ which satisfies the assumptions; then, by the universal property, we get an exact tensor functor $\widetilde{M}: \catF{C_1}\to \catF{C_2}$.
\item
The assumptions are satisfied if $\Ah=R\mod$ for a Dedekind domain $R$ where
$\Ah^\flat$ is the $\flat$-subcategory of projective finitely-generated $R$-modules, i.e., torsion free finitely-generated modules, and $M:C\to\Ah^\flat$ any tensor functor.  In particular this is true for $R=\Z$.
\item They are not satisfied for $\Ah=R\mod$ for a general Noetherian commutative ring $R$
and the subcategory of projective finitely-generated  $R$-modules, which is not a $\flat$-subcategory if the global dimension of $R$ is $>2$. See the example below.
\end{enumerate}
\end{rem}

\begin{ex}\label{ex:Zvier} Let $C$ be the category with objects $(\Z/4)^n$ for $n\geq 0$ and morphisms given by homomorphisms of abelian groups.

Let $\Ah=\Z/4\mod$. In this case is possible to compute all objects explicitly.
The functor $M\mapsto M^\vee=\Hom(M,\Z/4)$ is an antiequivalence of $C$ with itself. We have $C\mod\isom\Z/4\mod$ with $C\to\Z/4\mod$ given by $M\mapsto M^\vee$.
Hence
$\catF{C}$ is the category of finitely presented presheaves on $\Z/4\mod$. Objects are uniquely determined by the values of these presheaves on the groups $\Z/4$ and $\Z/2$.
Direct computation will show:
\begin{enumerate}
\item $\tensor$ is not biexact on $\catF{C}$.
\item The tensor functor $\catF{C}\to\Ah$ induced by the inclusion functor $C\to\Ah$ is not a tensor functor.
\end{enumerate}
By Auslander-Reiten theory, see e.g.~\cite[\S IV.6, p.~149]{ASS},
the simple objects of the category $\catF{C}$ have the form $(X,-)/{\rm rad}(X,-)$ for $X$ an indecomposable $\Z/4$-module.  So there are two simple objects, $S$ and $T$ say, and these are such that $S(\Z/4) =\Z/2$, $S(\Z/2) =0$ and $T(\Z/4) =0$, $T(\Z/2) =\Z/2$.  Noting the exact sequence $0 \to \Z/2 \xrightarrow{j} \Z/4 \xrightarrow{p} \Z/2 \to 0$ and considering the maps $(p,-)$ and $(j,-)$ in $\catF{C}$, it can be easily checked that ${\rm rad}(\Z/4,-) =(\Z/2,-)$ and that $(\Z/2,-)$ has length 2, with socle $S$.  The remaining indecomposable objects of $\catF{C}$ may then be computed (see, for example, \cite[4.3]{PreAxtFlat}):  there are 5 of them, all of them subquotients of the two representable functors.  They are $(\Z/4,-)$, $(\Z/2,-)$, the two simples $S$, $T$ and $(\Z/4,-)/S$.

Now consider the exact functor $\widetilde{\Z/4}: \catF{C} \to \Z/4\mod$.  This is evaluation of an object of $\catF{C}$, considered as a functor on $\Z/4\mod$, at $\Z/4$, hence is $0$ only on $T$ among those five indecomposables.  Therefore its kernel is the Serre subcategory which consists of direct sums of copies of $T$.  In order to compute $T\otimes T$, we apply the definition of the tensor product on $\catF{C}$ using the projective presentation $(\Z/4,-) \xrightarrow {(j,-)}(\Z/2,-) \xrightarrow{\pi_T} T \to 0$ of $T$
and, checking that $({\rm id}_{\Z/2}, -) \otimes (j,-) =0$, we obtain $T\otimes T = (\Z/2,-)$, which is not in the kernel of $\widetilde{\Z/4}$, so this is not a tensor functor.
As part of the computation of $T\otimes T$ one sees that $T\otimes (\Z/2,-) =(\Z/2,-)$.  So applying $T\otimes -$ to the monomorphism $(\Z/2,-) \xrightarrow{(p,-)} (\Z/4,-)$ gives $(\Z/2,-) \to T$ which is not monic, showing that $\otimes$ is not exact on $\catF{C}$.
\end{ex}

This implies that we cannot expect a different, exact, tensor product on $\catF{C}$ extending the tensor product on $C$ -- by the universal property the identity would have to be a tensor functor.

\subsection*{Tensor structures on $\catN{M}$}
Consider $(\Ah, \tensor)$ an abelian tensor category with a right-exact tensor product.

For the sake of exposition we now drop explicit reference to $\tensor$ if unnecessary.
\begin{prop} \label{tensor:quot}
Let $C$ be an additive tensor category, $\Ah$ an abelian tensor category with a right exact tensor product, and $M:C\to\Ah$ an additive tensor functor. Further assume that $M$ factors through a $\flat$-subcategory $\Ah^\flat \subset \Ah$ (see Definition~\ref{def:flat-subcat}).
\begin{enumerate}
\item
 Then $\catN{M}$ carries a canonical tensor structure such that the faithful exact  functor $\widetilde{M}:\catN{M}\to\Ah$ is a tensor functor.
\item If in addition, the tensor structures on $C$ and $\Ah$ are commutative and
the tensor functor is symmetric, then the tensor product on $\catN{M}$
is symmetric.
\item If in addition, the tensor structure on $C$ is rigid and the tensor product and the $\Hom$-functor on $\Ah$ are exact in both arguments, then the same is true for $\catN{M}$.
\end{enumerate}
\end{prop}
\begin{proof}
We need to check that the tensor functor on $\catF{C}$ (see Proposition~\ref{prop:tensor}) factors
via an induced tensor structure on $\catN{M}$. We have a commutative diagram
\[\begin{CD}
\catF{C}\times \catF{C}@>\otimes>> \catF{C}\\
@V{\widetilde{M}\times \widetilde{M}}VV@VV\widetilde{M}V\\
\Ah\times \Ah@>\otimes>>\Ah
\end{CD}\]
by Proposition~\ref{prop:right_exact}.
This implies that the kernel of $\catF{C}\to \Ah$ is a $\tensor$-ideal. Hence
the tensor product induces one on $\catN{M}$.
Associativity, unit and symmetry are immediate from the properties of the tensor structure on
$\catF{C}$.

We turn to rigidity. By assumption, every object $X$ of $C$ has a strong dual.
By the criterion formulated in \cite[Part I, IV, Proposition~1.1.9]{Lev} the existence of a dual for $X$ can be characterized by the existence of unit and counit maps satisfying some compatibilities. In particular,
this property is functorial, hence the image of $X$ in $\catN{M}$ also has a strong dual.
Consider the full subcategory of $\catN{M}$ consisting of objects with a strong
dual. It contains all objects in the image of $C$. Under our assumptions on $\Ah$, the tensor product on $\catN{M}$ is exact in both arguments
and hence the subcategory is closed under kernels and cokernels.
Hence it is an abelian subcategory of $\catN{M}$
containing the image of $C$, hence it agrees with $\catN{M}$
\end{proof}

\begin{prop}[Universal property]\label{prop:univ_prop_Nori}
Let $C$, $\Ah^\flat\subset \Ah$, and $M$ be as in Proposition~\ref{tensor:quot}.
In addition, let $\Bh$ be another abelian tensor category with $\flat$-subcategory $\Bh^\flat$, $N:C\to\Bh$ an additive tensor functor which factors through $\Bh^\flat$ and
$\phi:\Bh\to\Ah$ a faithful exact functor mapping $\Bh^\flat$ to $\Ah^\flat$
such that $\phi\circ N=M$:
\[\begin{xy}\xymatrix{
 &\catN{M}\ar[rd]^{\tilde{M}}\ar@{.>}[dd]\\
C\ar[ru]\ar[rd]_N\ar[rr]^{\hspace*{6ex}M}&&\Ah\\
&\Bh\ar[ru]_\phi
}\end{xy}\]
Then there exists a unique faithful exact tensor functor $\Phi:\catN{M}\to\Bh$ making the diagram commute.

In particular, the universal property characterises $(\catN{M},\tilde{M})$ uniquely up to unique equivalence of categories.
\end{prop}
\begin{proof}
The universal property of $\catF{C}$ (see Proposition~\ref{prop:right_exact}) gives us a similar statement,  but with
$\catF{C}$ instead of $\catN{M}$. The kernels of $\catF{C}\to\Ah$ and
$\catF{C}\to \Bh$ agree because $\phi$ is faithful and exact. Hence
$\catN{M}=\catN{N}$ and $\Phi=\tilde{N}$.
\end{proof}
By a simple trick that we learned from Arapura in \cite{Ar}, this can be upgraded to a more general one.

\begin{cor}[Generalised universal property]\label{cor:gen_univ_prop}
Let $C$, $\Ah^\flat\subset \Ah$, and $M$ be as in Proposition~\ref{tensor:quot}.
In addition, let $\Bh$ be another abelian tensor category with $\flat$-subcategory $\Bh^\flat$, $N:C\to\Bh$  an additive tensor functor which factors through $\Bh^\flat$.
Let $(\Bh',\Bh'^\flat)$ be a third abelian tensor category, $\phi:\Ah\to\Bh'$ and
$\psi:\Bh\to\Bh'$ faithful exact tensor functors respecting the $\flat$-subcategories. Finally, let $F:\psi\circ N\to \phi\circ M$ be an isomorphism of
functors:

\[\begin{xy}\xymatrix{
 &\catN{M}\ar[rd]^{\tilde{M}}\ar@{.>}[dd]\\
C\ar[ru]\ar[rd]_N\ar[rr]^{\hspace*{6ex}M}&&\Ah\ar[d]_{\phi}\\
&\Bh\ar[r]_\psi&\Bh'
}\end{xy}\]
Then there exists a faithful exact tensor functor $\Phi:\catN{M}\to\Bh$ making the diagram commute up to isomorphism of functors.
\end{cor}
\begin{proof}
Let $\Ch$ be tensor category with objects of the form
$(A,B,f)$ where $A\in\Ah$, $B\in\Bh$ and $f:\psi(B)\to\phi'(A)$.  This category is abelian with kernels and cokernels taken componentwise. We equip it with a right exact tensor product
$(A,B,f)\tensor (A',B',f')=(A\tensor A',B\tensor B',f\tensor f')$.
Let $\Ch^\flat$ be the subcategory of $(A,B,f)$ with $A\in\Ah^\flat$, $B\in\Bh^\flat$.

Let $N':C\to\Ch$ be the additive tensor functor $X\mapsto (M(X),N(X),F_X)$.
We apply the universal property of Proposition~\ref{prop:univ_prop_Nori} to
$N'$ and the forgetful functor $\phi:\Ch\to\Ah$. We define
$\Phi$ as $\widetilde{N}'$ composed with the forgetful functor to $\Bh$. In other
words, $\widetilde{N}'(X)=(\widetilde{M}(X),\widetilde{N}(X),F_X)$. The isomorphism of functors is given by $F_X$.
\end{proof}
\begin{ex}A possible application is
with $\Ah$, $\Bh$, $\Bh'$ the categories of $k$-vector spaces, $L$-vector spaces and $L'$-vector spaces, respectively, for field extensions $L'/k$ and $L'/L$.
\end{ex}

\begin{rem}This is a version of Nori's result on the tensor structure on his abelian category, see \cite[Proposition~8.1.5]{HMS}. It is much stronger in allowing
general abelian categories $\Ah$ as target.
In loc. cit. it was claimed that the original construction works
for functors $C\to R\Proj$ (where the latter is the category of finitely generated projective modules over a noetherian ring $R$). However, as Paranjape pointed out, the proof is only correct if kernels of of maps projective modules are projective, i.e., if the global dimension of $R$ is at most $2$.
\end{rem}

\section{Universal $\tensor$-representation}\label{sect:quiver}

We want to extend our results to representations of quivers. Given the results of the previous section, this means to extend tensor structures from a quiver to the additive category generated by it.

Recall from \cite[Def. 5.1.5]{Bo} or \cite{GR} the concept  of a quiver ``with relations'', i.e., a quiver (a collection of vertices and directed edges) with a set of commutativity conditions or linear relations between paths (= compositions of directed edges).
In this sense:

\begin{defn}\label{def:tensorquiver}
A \emph{$\tensor$-quiver} is a quiver $D$ with relations, including the following data
$(\id,\tensor,\alpha,\beta,\beta',\unit,u)$
\begin{enumerate}
\item for every vertex $v$ a distinguished self-edge $\id:v\to v$;
\item for every pair of vertices $(v,w)$ a vertex denoted $v\tensor w$ in $D$;
\item for every edge $e:v\to v'$ and vertex $w$ an edge $e\tensor\id:v\tensor w\to v'\tensor w$ and an edge  $\id\tensor e:w\tensor v\to w\tensor v'$;
\item for every pair of vertices $u,v$ a distinguished edge $\alpha_{u,v}:u\tensor v\to v\tensor u$;
\item for every triple of vertices $u,v,w$ a distinguished edge
$\beta_{u, v w}: u \otimes (v \otimes w)\to  (u \otimes v) \otimes w$ and also $\beta_{u, v w}': (u \otimes v) \otimes w\to  u \otimes (v \otimes w)$;
\item a distinguished vertex $\unit$;
\item for every vertex  distinguished edges $u_v:v\to\unit\tensor v$
and $u'_v:\unit\tensor v\to v$;
\end{enumerate}
and the relations
\begin{enumerate}
\item $\id_v\tensor\id_v=\id_{v\tensor v}$;
\item $\id_v=e_v$ where $e_v$ is the empty path for every vertex $v$;
\item \label{comp} $(e\tensor\id)\circ (\id\tensor e')=(\id\tensor e')\circ (e\tensor \id)$ for
all pairs of edges $e,e'$;
\item $\alpha_{v,w}\circ \alpha_{w,v}=\id$ for all vertices $v,w$;
\item $(\id\tensor\gamma)\circ\alpha=\alpha\circ (\gamma\tensor\id)$ and
$(\gamma\tensor\id)\circ \alpha=\alpha\circ(\id\tensor\gamma)$ for all edges $\gamma$;
\item $\beta_{u,vw}\circ \beta'_{uv,w}=\id$, $\beta'_{uv,w}\circ \beta_{u,vw}=\id$;
\item $\beta\circ(\gamma\tensor(\id\tensor\id))=((\gamma\tensor\id)\tensor \id)\circ\beta$ for all edges $\gamma$ and analogously in the second and third argument;
\item (pentagon axiom) for all vertices $x,y,z,t$ the relation
\[\begin{xy}\xymatrix{
x\tensor(y\tensor (z\tensor t))\ar[r]^{\beta}\ar[d]_{\id\tensor\beta}&(x\tensor y)\tensor (z\tensor t)\ar[r]^{\beta}&((x\tensor y)\tensor z)\tensor t\\
x\tensor((y\tensor z)\tensor t)\ar[rr]^{\beta}&&(x\tensor(y\tensor z))\tensor t\ar[u]_{\beta\tensor\id}
}\end{xy}\]
\item for all vertices $x,y,z$ the relation
\[\begin{xy}\xymatrix{
x\tensor(y\tensor z)\ar[r]^\beta\ar[d]_{\id\tensor\alpha}&(x\tensor y)\tensor z\ar[r]^\alpha&z\tensor(x\tensor y)\ar[d]^{\beta}\\
x\tensor(z\tensor y)\ar[r]^\beta&(x\tensor z)\tensor y\ar[r]^{\alpha\tensor \id}&(z\tensor x)\tensor y
}\end{xy}\]
\item $u_v\circ u'_v=\id$ and $u'_v\circ u_v=\id$ for all vertices $v$;
\item for all edges $e:v\to v'$ the relation
\[\begin{xy}\xymatrix{
 v'\ar[r]^{u}&\unit\tensor v'\\
v\ar[u]^e\ar[r]^{u}&\unit\tensor v\ar[u]_{\id\tensor e}
}\end{xy}\]
\end{enumerate}
\end{defn}
\begin{rem}
This data is modeled after the notion of a commutative product structure on a diagram with identities, see
\cite[Def. 8.1.3]{HMS} and  the variant in loc.cit. Remark 8.1.6.
The axioms for the associativity and commutativity constraint and unitality are the usual ones for a commutative tensor category, see in \cite[\S 1]{DMT}.

It is more general than the notion of a monoidal quiver introduced by Brugui\`eres, see \cite[Section~5.2]{Br}.
\end{rem}
Recall \cite{BVP} where a universal representation $\Delta: D \to  \catF{D}$ is constructed for any quiver $D$. It is given by the composition
\[ \Delta: D\to\pathcat(D)\to\Z D\to\Z D^+\to\catF{\Z D^+}=\catF{D}\]
where (in the notation of \cite[\S 1]{BVP}) $\pathcat(D)$ is the path category,
$\Z D$ the preadditive enrichment of $\pathcat(D)$ and
$\Z D^+$ its additive completion.

We now repeat the same chain with tensor categories. Let $(D,\tensor)$ be a $\tensor$-quiver. We define the $\tensor$-path category $\pathcat(D)^\tensor$
as the quotient of the path category by the relations of $(D,\tensor)$. We define
\[ \tensor:\pathcat(D)\times\pathcat(D)\to\pathcat(D)\]
on objects as prescribed by the tensor structure.
Let $\Gamma=\gamma_1\circ\dots\circ \gamma_n$, $\Delta=\delta_1\circ\dots\circ \gamma_m$ be paths. We define
\[ \Gamma\tensor\Delta=(\gamma_1\tensor\id)\circ\dots\circ (\gamma_n\tensor\id)\circ(\id\tensor\delta_1)\circ\dots\circ(\id\tensor\delta_m).\]
E.g. for $n=m=1$ and $\gamma:v\to v'$, $\delta:w\to w'$, we have
\[\begin{xy}\xymatrix{
  v\tensor w\ar[r]^{\id\tensor\delta}\ar[d]^{\gamma\tensor\id}\ar[rd]^{\gamma\tensor\delta}&v\tensor w'\ar[d]^{\gamma\tensor\id}\\
  v'\tensor w\ar[r]_{\delta\tensor\id}&v'\tensor w'
}\end{xy}\]
where we have by definition set the diagonal to be the path via the top right corner. In $\pathcat(D)^\tensor$ this agrees with the path via the bottom left corner because of the relation (\ref{comp}).
\begin{lemma}Let $(D,\tensor)$ be a $\tensor$-quiver. Then $\pathcat(D)^\tensor$ is
a tensor category.
\end{lemma}
\begin{proof} Property (\ref{comp}) of a tensor structure ensures that
$\tensor$ is a functor on $\pathcat(D)^\tensor$. The other axioms make sure that
the commutativity constraint $\alpha$ and the associativity constraint $\beta$
are isomorphisms and satisfy the properties of a commutative tensor category.
The relations on $u_v$ ensure that $v\to \unit\tensor v$ is an isomorphism and
the functor $\unit\tensor-$ is an equivalence of categories.
\end{proof}

\begin{defn}
Let $(D,\tensor)$ be a $\tensor$-quiver. We put
$\Z D^\tensor$ and $\Z D^{\tensor,+}$ the preadditive and additive hull of
$\pathcat(D)^\tensor$. Denote by $\catF{D}^\tensor$ Freyd's abelian category of $\Z D^{\tensor,+}$.
\end{defn}

\begin{prop}\label{path:tensor}
$\Z D^\tensor$, $\Z D^{\tensor,+}$ and $\catF{D}^\tensor$ with the bilinear extension of $\tensor$ are commutative tensor categories.  The canonical functor $\Z D^+\to \Z D^{\tensor,+}$ induces a Serre quotient $\pi:\catF{D}\to \catF{D}^\tensor$. \end{prop}
\begin{proof}
The statements on the additive and preadditive category are obvious. The
statement on the abelian category is Proposition~\ref{prop:tensor}. The claim on the Serre quotient is granted by the following general fact.
\end{proof}
\begin{lemma}Let $D_1$ be a quiver with relations and $D$ the underlying quiver. Then $\pi:\catF{D}\to\catF{D_1}$ is a Serre quotient.
\end{lemma}
This is well-known but for the convenience of the reader we give the simple argument directly.
\begin{proof}
Consider $\catF{D}/\Ker(\pi)$.
By construction
this is an exact subcategory of $\catF{D_1}$, hence it remains to check that the inclusion is full and essentially surjective.

The quiver $D$ has a canonical representation in $\catF{D}/\Ker\pi$. All relations in $D_1$ are satisfied, hence it is even a representation of $D_1$.
By the universal property
this yields an exact functor $\catF{D_1}\to\catF{D}/\Ker\pi$.
By the uniqueness part of the universal property, its composition with the
inclusion into $\catF{D_1}$ is isomorphic to the identity. In particular, the
inclusion is full and essentially surjective, hence an equivalence of categories.
\end{proof}

We now turn to the universal property. The obvious approach is to
consider representations $T:D\to \Ah$ where all relations in $D$ are mapped to
identities in $\Ah$. However, this is too rigid for most applications.
We follow the approach of \cite[Definition~8.1.3]{HMS}

\begin{defn}Let $D$ be a $\tensor$-quiver, $\Ah$ a commutative tensor category.
A \emph{tensor representation} or \emph{$\tensor$-representation} for short, is a representation
$T:D\to\Ah$ of the underlying quiver
together with the choice of an isomorphism
$\kappa_0:\unit\to T(\unit)$ and of
natural isomorphisms $$\kappa :T(u) \otimes T(v)\xrightarrow{\simeq} T(u \otimes v)$$ for all vertices $u,v \in D$, 
functorial  in each variable and compatible with the associativity and commutativity constraints and the unit in the obvious way.
\end{defn}

\begin{prop}\label{path:univtensor}
Let $(D,\tensor)$ be a $\tensor$-quiver.
\begin{enumerate}
\item $D\to\pathcat(D)^\tensor$ is the universal $\tensor$-representation into a commutative tensor category.
\item $D\to \Z D^{\tensor,+}$ is the universal $\tensor$-representation into an additive commutative tensor category.
\end{enumerate}
\end{prop}
\begin{proof}The universal properties for $\pathcat(D)^\tensor$ and
$\Z D^{\tensor,+}$ are obvious.
\end{proof}

\begin{thm} \label{thm:abelian_tensor}
Let $(D,\tensor)$ be a $\tensor$-quiver.
\begin{enumerate}
\item  The natural assignment $\Delta^\tensor:D\to\catF{D}^\tensor$ is a $\tensor$-representation into
an abelian tensor category with right-exact tensor product.
\item \label{it:sub}It takes
values in the subcategory $(\catF{D}^\tensor)^\flat$
of Definition~\ref{defn:subcat}. Moreover, this is a $\flat$-subcategory (see Definition~\ref{def:flat-subcat}). 
\item The category $\catF{D}^\tensor$ is universal with this property.
\end{enumerate}
In detail:
Let $T : D\to \Ah$ be a $\tensor$-representation via $\kappa$ in an abelian tensor category with a right exact tensor, which factors through a $\flat$-subcategory $\Ah^\flat\subseteq \Ah$. Then there is an induced exact tensor functor $\widetilde{M}^\tensor :\catF{D}^\tensor\to \Ah$.
\end{thm}
\begin{proof}
 $\Ah =\catF{D}^\tensor$ is an abelian tensor category by Proposition~\ref{path:tensor} and $\Delta^\tensor$ is a $\tensor$-representation by construction.
It factors via the additive category $\Z D^{\tensor,+}$. Hence property (\ref{it:sub}) follows from Proposition~\ref{prop:tensor}. To see the second statement, note that if $(B,-) \xrightarrow{(f,-)} (A,-)$ is a morphism in $(\catF{D}^\tensor)^\flat$ then its kernel is $(C,-) \xrightarrow{(g,-)} (B,-)$ where $B\xrightarrow{g} C$ is the cokernel of $A\xrightarrow{f} B$.

The induced functor $M^\tensor:\Z D^{\tensor,+} \to \Ah$ satisfies the assumptions in Proposition~\ref{prop:right_exact}. Thus it induces the tensor functor $\widetilde{M}^\tensor: \catF{D}^\tensor\to \Ah$ such that $T=\widetilde{M}^\tensor\Delta^\tensor$.
\end{proof}

Recall the universal representation theorem stated in \cite{BVP}. For $T : D\to \Ah$ any representation of a quiver in an abelian category $\Ah$ there is an induced additive functor $$M : \Z D^{+}\to \Ah$$ and a corresponding $\widetilde{M} : \catF{D}\to \Ah$ in such a way that
$$\widetilde{T}:D\to \catN{T} := \catN{M} = \catF{D}/\Ker \widetilde{M}$$
is the induced universal representation (see \cite[\S 1.3]{BVP}).
For a $\tensor$-quiver $D$, together with a $\tensor$-representation $T$ in an abelian  tensor category $\Ah$, as in Theorem \ref{thm:abelian_tensor}, we have now constructed a factorisation via
an exact tensor functor $\widetilde{M}^\tensor$ on $\catF{D}^\tensor$. Hence  we get a tensorial refinement of the universal representation theorem. This also implies the existence of a tensor structure on the universal abelian category $\catN{T}$ attached to the representation. Note that this is really $\catN{T}$; in contrast to $\pathcat(D)^\tensor$ etc.~no $\tensor$-adornment is needed.

\begin{thm} \label{thm:Norigd}
Let $T : D\to \Ah$ be a representation in an abelian tensor category with a right
exact tensor, which factors through a $\flat$-subcategory $\Ah^\flat\subseteq \Ah$, with the following additional properties:
\begin{enumerate}
\item[\emph{(i)}] $(D,\tensor)$ is a $\tensor$-quiver and
\item[\emph{(ii)}] $T$ is a $\tensor$-representation in $\Ah^\flat\subseteq \Ah$ via $\kappa$.
\end{enumerate}
Then Nori's universal abelian category $\catN{T}$ carries a right exact tensor product and $\widetilde{M}:\catN{T}\to\Ah$ is a tensor functor (here $M$ is the additive functor induced by $T$ and $\widetilde{M}$ is the faithful exact functor induced by $M$, see also Proposition \ref{tensor:quot}).
It is universal among such representations into abelian tensor categories $\Bh$ compatible  with $T$ (cf.~the statement of Proposition \ref{prop:univ_prop_Nori}) via
a faithful exact tensor functor $\Bh\to \Ah$.
\end{thm}
\begin{proof}
By the universal property in Theorem \ref{thm:abelian_tensor}, there is a canonical exact tensor functor
$\widetilde{M}^\tensor:\catF{D}^\tensor\to \Ah$. Hence $\Ker\widetilde{M}^\tensor$ is a Serre subcategory and  a tensor ideal. Denoting by $\catN{T}^\tensor$ the Serre quotient $\catF{D}^\tensor/\Ker\widetilde{M}^\tensor$ we have obtained a tensor category with the universal property as claimed. Furthermore, by the universal property of $\catN{T}$, there is also an exact faithful functor
$$\catN{T}\to\catN{T}^\tensor$$
 We claim that it is an equivalence of abelian categories.
The canonical additive functor $\Z D^{+}\to \Z D^{\tensor,+}$ induces an exact functor $\pi: \catF{D}\to\catF{D}^\tensor$ such that $\widetilde{M}^\tensor\circ\pi = \widetilde{M}$ by the uniqueness in the universal property of Freyd's construction (see \cite[Thm. 1.1]{BVP}).
The faithful exact functor
$\bar{\pi}: \catF{D}/\Ker\pi\xrightarrow{\simeq} \catF{D}^\tensor$
is an equivalence by Proposition \ref{path:tensor}.

Thus, the composition $\catF{D} \xrightarrow{\pi} \catF{D}^\tensor \to \catF{D}^\tensor/\Ker(\widetilde{M}^\tensor)$ is essentially surjective and is equivalent to the composition $\catF{D} \to \catN{T} \to \catN{T}^\tensor$ since they have equivalent compositions with the faithful functor $\catN{T}^\tensor \to \Ah$.  So $\catN{T}\to\catN{T}^\tensor$ also is essentially surjective hence an equivalence.
\end{proof}
\begin{rem}
The universal property can be upgraded analogously to Corollary~\ref{cor:gen_univ_prop}.
\end{rem}

\begin{rem}In the special case where
$\Ah$ is the category of finitely generated modules over a Dedekind domain, this
gives back Nori's original result as formulated for example in \cite{HMS}.
The same case (actually in the more restrictive setting of monoidal quivers) is also handled by Brugui\`eres in \cite[Theorem~3]{Br}.  His conditions P1 and
P2 are analogous to our factorisation via $\Ah^\flat$. Both \cite{HMS} and
\cite{Br} are based on the explicit description of the universal abelian category as comodules or modules.
\end{rem}

\subsection*{Signs}
In many cases, notably in Nori's original application, we do not start with a
tensor representation but with a tensor representation with signs.
We explain the necessary modifications, following again the approach of
\cite[Def. 8.1.3]{HMS}.

\begin{defn}\label{defn:graded}A \emph{graded quiver} is a quiver together with a function
$|\cdot|$ assigning to each vertex a degree in $\Z/2\Z$. For an edge $e:v\to w$ we put $|e|=|w|-|v|$. A \emph{graded $\tensor$-quiver} is a graded quiver
together with the data of $\tensor$-quiver
such that $|v\tensor w|=|v|+|w|$ and  $|\unit|=0$. The relations are the same as for a $\tensor$-quiver, except for relation (\ref{comp}) which is replaced by
\begin{description}
\item[(\ref{comp}')]$(e\tensor\id)\circ (\id\tensor e')=(-1)^{|e||e'|}(\id\tensor e')\circ (e\tensor \id)$ for
all pairs of edges $e,e'$;
\end{description}
\end{defn}
The grading on $D$ induces gradings on $\pathcat(D)$, $\Z D$, and $\Z D^+$. In the
case of the additive hull this means that every object is equipped with  a decomposition
into an even and an odd part. Note that morphisms are \emph{not} required to preserve the degree.
Recall that part of the data of a $\tensor$-quiver is the choice of edges
$\alpha_{v,w}:v\tensor w\to w\tensor v$.

\begin{defn}\label{Def:absign}
Let $(D,\tensor)$ be a graded $\tensor$-quiver.
\begin{enumerate}
\item We define
$\Z D^{\tensor,\sign}$ as the quotient of the category $\Z D$ modulo the relations of a graded $\tensor$-quiver. It is equipped
with tensor product $\tensor^\sign$ which agrees with $\tensor$ on objects and for morphisms $\gamma:v\to v',\delta:w\to w'$
\[ \gamma\tensor^\sign \delta=(-1)^{|\gamma||w|}\gamma\tensor\delta,\]
with associativity constraint $\beta_{u,vw}^\sign=\beta_{u,vw}$ and
commutativity constraint given by
\[\alpha_{v,w}^\sign=(-1)^{|v||w| }\alpha_{v,w}: v\tensor w\to w\tensor v\]
for all objects $v,w$.
\item Let $\Z D^{\tensor,\sign,+}$ be the category $\Z D^{\tensor,+}$ with
tensor structure given by the additive extension from $\Z D^{\tensor,\sign}$.
\item Set $\catF{D}^{\tensor,\sign}=\catF{\Z D^{\tensor,\sign,+}}$
for the universal abelian category attached to $\Z D^{\tensor,\sign,+}$.
\end{enumerate}
\end{defn}
\begin{rem}Note that $\Z D^{\tensor}$ is different from $\Z D^{\tensor,\sign}$ even as an additive category.
\end{rem}

\begin{lemma}$\Z D^{\tensor,\sign}$ and $\Z D^{\tensor,\sign,+}$ are well-defined tensor categories.
\end{lemma}
\begin{proof} It suffices to consider $\Z D^{\tensor,\sign}$. We have to check that $\tensor^\sign$ satisfies the axioms of a commutative tensor category.
Condition (\ref{comp})' ensures functoriality of $\tensor^\sign$.
It is tedious but straightforward that $\beta$ and $\alpha$ are functorial. E.g. for $\gamma:x\to x'$, $\delta:y\to y'$ the diagram reads
\[\begin{xy}\xymatrix{
x\tensor y\ar[d]_{(-1)^{|\gamma||y|}\gamma\tensor\delta}\ar[r]^{(-1)^{|x||y|}\alpha}&y\tensor x\ar[d]^{(-1)^{|\delta||x|}\delta\tensor\gamma}\\
x'\tensor y'\ar[r]^{(-1)^{|x'||y'|}\alpha}&y'\tensor x'
}\end{xy}\]
It does not commute on the level of $\pathcat(D)$.
In order to check that it commutes in $\Z D^{\tensor,\sign}$, it is enough to
to treat the two special cases $\gamma=\id$ or $\delta=\id$ separately because
$(\gamma,\delta)=(\gamma,\id)\circ(\id,\delta)$. In each of these cases the diagram commutes in $\pathcat (D)$.

The pentagon axiom (concerning associativity) holds because it is a relation on $D$ and no signs are involved.
Unitality is preserved because $\unit$ is of degree $0$.
The hexagon axiom reads
\[\begin{xy}\xymatrix{
x\tensor(y\tensor z)\ar[r]^\beta\ar[d]_{\id\tensor(-1)^{|y||z|}\alpha}&(x\tensor y)\tensor z\ar[r]^{(-1)^{(|x|+|y|)|z|}\alpha}&z\tensor(x\tensor y)\ar[d]^{\beta}\\
x\tensor(z\tensor y)\ar[r]^\beta&(x\tensor z)\tensor y\ar[r]^{(-1)^{|x||z|}\alpha\tensor \id}&(z\tensor x)\tensor y
}\end{xy}\]
It commutes because the hexagon axiom holds for $\tensor$.
\end{proof}

Again, we turn to representations. Following \cite[Def. 8.1.3]{HMS}:

\begin{defn}Let $(D,\tensor)$ be a graded $\tensor$-quiver. Let
$\Ah$ be an additive commutative tensor category. A \emph{graded tensor representation} of $(D,\tensor)$ is a representation $T:D\to\Ah$ of the underlying quiver together
with the choice of an isomorphism
$\kappa_0:\unit\to T(\unit)$ and of natural isomorphisms $$\kappa :T(u) \otimes T( v)\xrightarrow{\simeq} T(u\otimes v)$$  for all vertices $u,v \in D$,
functorial  in each variable and compatible with the associativity constraint and the unit in the obvious way and such that
\begin{enumerate}
\item for all vertices $v,w$
\[\begin{xy}\xymatrix{
 T( v\tensor w)\ar[r]^{T(\alpha)}& T(w\tensor v)\\
 T(v)\tensor T(w)\ar[u]^{\kappa}\ar[r]&T(w)\tensor T(v)\ar[u]_{\kappa}
}\end{xy}\]
commutes where the bottom arrow is $(-1)^{|v| |w|}$ times the commutativity constraint in $\Ah$;
\item for all edges $\gamma:v\to v'$ and vertices $w$
\[\begin{xy}\xymatrix{
 T(v\tensor w)\ar[r]^{T(\gamma\tensor\id)}&T(v'\tensor w)\\
 T(v)\tensor T(w)\ar[u]^{\kappa}\ar[r]^{T(\gamma)\tensor\id}&T(v')\tensor T(w)\ar[u]_{\kappa}
}\end{xy}\]
commutes up to the factor $(-1)^{|\gamma||w|}$.
\item for all edges $\gamma:v\to v'$ and vertices $w$
\[\begin{xy}\xymatrix{
 T(w\tensor v)\ar[r]^{T(\id\tensor\gamma)}&T(w\tensor v')\\
 T(w)\tensor T(v)\ar[u]^{\kappa}\ar[r]^{\id\tensor T(\gamma)}&T(w)\tensor T(v')\ar[u]_{\kappa}
}\end{xy}\]
commutes (without signs).
\end{enumerate}
\end{defn}
The following is a graded analogue of Proposition~\ref{path:univtensor} (2).
\begin{prop}
Let $(D,\tensor)$ be a graded $\tensor$-quiver. The natural map $D\to \Z D^{\tensor,\sign, +}$  is the universal graded $\tensor$-representation of $(D,\tensor)$. In detail: it
is a graded $\tensor$-representation and if $T:D\to\Ah$ is a graded tensor representation in an additive commutative tensor category $\Ah$ then $T$ factors uniquely through an induced additive tensor functor as shown
\[\begin{xy}\xymatrix{
D\ar[r]\ar[dr]_{T} &\Z D^{\tensor,\sign,+}\ar[d]^{M^{\tensor,\sign}}\\
&\Ah
}\end{xy}\]
\end{prop}
\begin{proof}The argument is the same as in the ungraded case. Relation
(\ref{comp})' is forced by the signs in the graded tensor representation.
\end{proof}

Now consider the category $\catF{D}^{\tensor,\sign}$ as in Definition~\ref{Def:absign} (3).
\begin{thm}The category $\catF{D}^{\tensor,\sign}$ satisfies the graded analogue of Theorem~\ref{thm:abelian_tensor}.
\end{thm}
\begin{proof}As in in the ungraded case.
\end{proof}
Finally:
\begin{thm}\label{thm:Nori_graded}
Let $T:D\to\Ah$ be a representation in an abelian tensor category with a right exact tensor, which factors through a $\flat$-subcategory
$\Ah^\flat\subset\Ah$ with the following additional properties
\begin{enumerate}
\item $(D,\tensor)$ is a graded $\tensor$-quiver and
\item $T$ is a graded $\tensor$-representation in $\Ah^\flat\subset\Ah$ via $\kappa$.
\end{enumerate}
Then Nori's universal abelian category $\catN{T}$ carries a right exact tensor product and $\widetilde{M}:\catN{T}\to\Ah$ is a tensor functor (here $M$ is the additive functor induced by $T$ and $\widetilde{M}$ is the faithful exact functor induced by $M$, see also Proposition~\ref{tensor:quot}). It is universal among such representations into abelian tensor categories $\Bh$ compatible with $T$ (cf.~Proposition \ref{prop:univ_prop_Nori}) via a faithful exact tensor functor $\Bh\to\Ah$.
\end{thm}
\begin{proof}Compare with the proof of Theorem~\ref{thm:Norigd}. If $T$ is such a graded tensor representation we get $M^{\tensor,\sign} :\Z D^{\tensor,\sign,+}\to \Ah$ and also an induced exact tensor functor $\widetilde{M}^{\tensor,\sign}: \catF{D}^{\tensor,\sign}\to \Ah$.  Denote by $\catN{T}^{\tensor,\sign}$ the quotient of $\catF{D}^{\tensor,\sign}$ by the kernel of $\widetilde{M}^{\tensor,\sign}$. We have that $\catN{T}\to \catN{T}^{\tensor,\sign}$ is an equivalence.\end{proof}

\begin{rem}
Again, the universal property can be upgraded analogously to Corollary~\ref{cor:gen_univ_prop}.
\end{rem}

\section{Homological functors}\label{sec:kuenneth}
%\label{sec:homol}

We return to the case of additive categories, but specialise further by considering triangulated categories and homological functors.

\begin{prop}[Neeman {\cite[Theorem~5.1.18]{N}}]
Let $\Th$ be a triangulated category. Then there is an abelian category
$\neeman{\Th}$ and a homological functor $[-]:\Th\to\neeman{\Th}$ such that
every homological functor $\Th\to\Ah$ into an abelian category factors
uniquely via an exact functor $\neeman{\Th}\to\Ah$.
\end{prop}
Neeman uses the notation $\Ah(\Th)$, which we have reserved for Nori's abelian category. By construction, $\neeman{\Th}$ is the subcategory of finitely presented objects in the category of
presheaves of abelian groups on $\Th$. It is obtained from
the image of the Yoneda functor by adding all cokernels. Hence every object of $\neeman{\Th}$ is the cokernel of a morphism of objects in the image of the Yoneda functor and this is a projective resolution. Note, however, see \cite[\S 5.2, Appx. C]{N}, that the category $\neeman{\Th}$ is typically not well-powered.  Our constructions and results in Section \ref{sec:universal_tensor} do not require the initial categories $C$ to be well-powered, so those results apply here.

\begin{prop}\label{prop:tensor_triangle}
Let $\Th$ be a tensor triangulated category. Then $\neeman{\Th}$ carries
a right exact tensor product. If $\Ah$ is an abelian tensor category with right exact tensor and
$\Th\to\Ah$ is a homological tensor functor, then the natural functor
$\neeman{\Th}\to\Ah$ is a tensor functor.
\end{prop}
\begin{proof}We extend from representable objects to cokernels by using
projective resolutions $[B]\to[A]\to X\to 0$, where $[A]$ denotes the image of $A$ in $\neeman{\Th}$. The arguments are dual
to the ones used in the first part of the proof of Proposition~\ref{prop:mult}.
The associativity constraint etc. are constructed in the same
way as in the proof of Proposition~\ref{prop:tensor}. The compatibility of
$\neeman{\Th}\to\Ah$ with the tensor structure holds for representable arguments and extends to cokernels by right exactness.
\end{proof}

\begin{rem}
\begin{enumerate}
\item This was already proved by Balmer, Krause and Stevenson in
\cite[Proposition~A.14]{BKS} for compactly generated tensor triangulated categories for the smaller category of all presheaves on $\Th^c$ which is universal for all homological functors commuting with colimits, see \cite[Section~2]{K}.
\item Applying the universal property of $\catF{\Th}$ to $\Th\to\neeman{\Th}$, we obtain an exact functor $\catF{\Th}\to\neeman{\Th}$ but it is not clear whether this is a tensor functor.  This may well be false since the kernel of a map between representable functors in $\neeman{\Th}$ might not be tensor-flat.
\end{enumerate}
\end{rem}

\begin{prop} \label{tensor:quot_triang}
Let $\Th$ be tensor triangulated category, $\Ah$ an abelian tensor category with a right exact tensor product, and $M:\Th\to\Ah$ a homological functor and tensor functor.
\begin{enumerate}
\item
 Then $\catN{M}$ carries a canonical right exact tensor structure such that the faithful exact  functor $\widetilde{M}:\catN{M}\to\Ah$ is a tensor functor.
\item If in addition, the tensor structures on $\Th$ and $\Ah$ are commutative and
the tensor functor is symmetric, then the tensor product on $\catN{M}$
is symmetric.
\item If in addition, the tensor structure on $\Th$ is rigid and the tensor product and the $\Hom$-functor on $\Ah$ are exact in both arguments, then the same is true for $\catN{M}$.
\end{enumerate}
\end{prop}
\begin{proof}Same proof as for Proposition~\ref{tensor:quot}, but with
$\catF{\Th}$ and the tensor product of Proposition~\ref{prop:tensor} replaced with
$\neeman{\Th}$ and the tensor product of Proposition~\ref{prop:tensor_triangle}.
\end{proof}

\begin{rem}
If both Proposition~\ref{tensor:quot} and Proposition~\ref{tensor:quot_triang} apply, then by the universal property of Proposition~\ref{tensor:quot_triang} the tensor structures agree because they agree on objects in the image of
$\Th$.
\end{rem}

\subsection*{K\"unneth components}

We now consider the following situation modeled for the application to Nori motives.
Let $\Th$ be a triangulated category, $\Ah$ an abelian category and
$R:\Th\to D^b(\Ah)$ an exact
functor. We abbreviate
$H^i_R:=H^i\circ R$ and $H^*_R:=\bigoplus H^i_R$. The latter
is understood with values in $\gr\Ah$. Let $\catN{H^*_R}$ be the universal abelian category defined by $H^*_R$ and $\catN{H^0_R}$ that defined by $H^0_R$. The commutative diagram
\[ \begin{xy}\xymatrix{
 \Th\ar[rd]_{H^0_R}\ar[r]^{H^*_R}&\gr\Ah\ar[d]^{(-)^0}\\
&\Ah}\end{xy}\]
induces a functor $\catN{H^*_R}\to \catN{H^0_R}$. We also have $\widetilde{H^*_R}: \catN{H^*_R}\to \gr\Ah$.
\begin{defn}In the above situation let
$\catNsub{0}{H^*_R}\subset \catN{H^*_R}$ be the full
subcategory of objects $X\in \catN{H^*_R}$ with $\widetilde{H^*_R}(X)\in \gr\Ah$ concentrated in degree $0$.
\end{defn}
The subcategory is abelian and closed under subquotients and extensions.

\begin{rem}\label{rem:strategy}We are interested in the case where $\Th$ is a triangulated tensor category, $\Ah$ an abelian tensor category with an exact tensor product and
$R$ a tensor functor. Then $H^*_R$ is a tensor functor, but $H^0_R$ is not.
Hence while $\catN{H_R^*}$ is a tensor category by the results of Section~\ref{sec:universal_tensor}, this does not follow for $\catN{H^0_R}$.
It is, however, true for $\catNsub{0}{H^*_R}$. In good cases, it
will be equivalent to $\catN{H^0_R}$, giving the latter the tensor structure that we want.
\end{rem}

\begin{prop}\label{prop:iso_derived}Let $\Th$, $\Ah$ and $R$ be as above. Assume in addition that
$R$ can be lifted to an exact functor
\[ R:\Th\to D^b(\catNsub{0}{H^*_R}).\]
Then the natural functor
\[ \catNsub{0}{H^*_R}\to \catN{H^0_R}\]
is an equivalence of categories.
\end{prop}
\begin{proof}We abbreviate $\Ah':=\catNsub{0}{H^*_R}$.
By assumption, there is a commutative diagram
\[\begin{xy}\xymatrix{
\Th\ar[r]\ar[rd]&D^b(\Ah')\ar[r]\ar[d]^{H^0} &D^b(\Ah)\ar[d]^{H^0}\\
&\Ah'\ar[r]&\Ah
}\end{xy}\]
The functor $\widetilde{H^0}_R:\catN{H^0_R}\to\Ah$ is faithful and exact by construction.
The same is true for  $\widetilde{H^*_R} :\Ah'\to \gr\Ah$. By definition, this functor
takes values in degree $0$, hence $\Ah'\to\Ah$ is also faithful and exact.
This implies that the universal categories defined by $H^0:\Th\to\Ah'$ and
$H^0_R:\Th\to\Ah$ agree. This gives $\catN{H^0_R}\to\Ah'$ inverse
to the inclusion.
\end{proof}

\begin{cor}\label{cor:tensor_case}Let $\Th$ be a tensor triangulated category. Let
$\Ah$ be an abelian tensor category with an exact tensor product. Let
$R:\Th\to D^b(\Ah)$ be a tensor triangulated functor. Assume in addition, that
$R$ factors via $D^b(\catNsub{0}{H^*_R})$. Then
$\catN{H^0_R}$ carries a natural tensor structure such that
$\catN{H^0_R}\to\Ah$ is a tensor functor. If the tensor product on $\Th$
is rigid and $\Hom_\Ah$ exact in both variables, then the tensor product on $\catN{H^0_R}$ is rigid as well.
\end{cor}
\begin{proof}Combining Proposition~\ref{prop:iso_derived} with the strategy of Remark~\ref{rem:strategy} gives the tensor structure.
	If the tensor product on $\Th$ is rigid and $\Hom_\Ah$ exact, then by Proposition~\ref{tensor:quot}  tensor product on $\catN{H^*_R}$ is rigid as well. Hence every object $X$ of $\catNsub{0}{H^*_R}$ has a dual $X^\vee$ in $\catN{H^*_R}$. The object $X^\vee$ is actually in $\catNsub{0}{H^*_R}$, as we can test by applying the forgetful functor to $\gr\Ah$.
\end{proof}

\begin{rem}
\begin{enumerate}
\item The use of the \emph{bounded} derived category in the above argument is
not very important. We can drop the assumption, if arbitrary direct sums exist in
$\Ah$. This is needed in order to write down the K\"unneth formula or, equivalently, the tensor structure on $D(\Ah)$.
\item We may also replace $D^b(\Ah)$ by a tensor triangulated category equipped with a  $t$-structure (compatible with the tensor structure) with heart $\Ah$ without any change in the arguments.
\end{enumerate}
\end{rem}

\subsection*{Integral coefficients}
What we have done so far does not apply to $\Ah=\Z\mod$ because its tensor product is not exact.
However, there is a version of the above criterion for integral coefficients.

Let $\Th$ be a triangulated category. Let $\Ah$ be an abelian tensor category with a right exact tensor product such that its derivation on $D^b(\Ah)$ exists. Let $\Ah^\flat\subset\Ah$ be a $\flat$-subcategory as in Definition~\ref{def:flat-subcat}. Let $R:\Th\to D^b(\Ah)$ be a tensor functor. Note that
$H^*_R:\Th \to \gr\Ah$ is no longer a tensor functor because $H^*:D^b(\Ah)\to\gr\Ah$ is not.
However:
\begin{lemma}In this situation,
let $\Th^\flat\subset \Th$ be the full subcategory of objects with $H^*_R$ in $\gr\Ah^\flat$. Then
$\Th^\flat$ is a tensor category and
\[ H^*_R|_{\Th\flat}:\Th^\flat\to\gr\Ah^\flat\]
is a tensor functor satisfying the assumptions of the universal property in Proposition~\ref{prop:right_exact}.
\end{lemma}
\begin{proof}Obviously $\gr\Ah^\flat\subset\gr\Ah$ consists of flat objects and is closed under kernels. It remains to check the claim on the tensor functor with $\Th=D^b(\Ah)$. This amounts to the naive K\"unneth formula for these objects. The subcategory $\Th^\flat$ is stable under the canonical truncation functor and shift. Hence it suffices to check the formula for objects of $\Ah^\flat\subset \Th^\flat$. They are flat, hence the derived tensor product agrees with the tensor product in $\Ah^\flat$.
As a byproduct of the formula we see that $\Th^\flat$ is stable under the derived tensor product.
\end{proof}
We now replace $\catN{H^*_R}$ by $\catN{H^*_R|_{\Th^\flat}}$ and set as before
$\catNsub{0}{H^*_R|_{\Th^\flat}}$ to be the subcategory of objects concentrated in degree $0$.

\begin{cor}\label{cor:tensor_case_right}Let $\Th$ be a tensor triangulated category. Let
$\Ah$ be an abelian tensor category with a right exact tensor product. Let $\Ah^\flat\subset\Ah$
be a $\flat$-subcategory and assume that the derived tensor product exists on $D^b(\Ah)$. Let
$R:\Th\to D^b(\Ah)$ be a tensor triangulated functor.
 Let $\Th^\flat$ and
$\catNsub{0}{H^*_R|_{\Th^\flat}}$ be as above.

Assume in addition, that
$R$ factors via $D^b(\catNsub{0}{H^*_R|_{\Th^\flat}})$. Then
$\catN{H^0_R}$ carries a natural tensor structure such that
$\catN{H^0_R}\to\Ah$ is a tensor functor.
\end{cor}

\section{Nori motives}\label{sec:motives}
Recall the original definition of Nori. Let $k$ be field, $\sigma:k\to\C$ an embedding.
Let $\Sch_k$ be the category of schemes which are separated and of finite type over the field $k$. Let $D^{\Nori}$ be Nori's quiver on ${\rm Sch}_k$ having vertices $(X,Y,n)$ where $Y\subseteq X$ is a closed subscheme and $n\in \Z$ and edges $(X', Y', n)\to (X, Y, n)$ for each morphism $f: X\to X'$ in $\Sch_k$  such that $f (Y)\subseteq Y'$, and an additional edge $(Y, Z, n)\to (X, Y, n+1)$ for $Z\subseteq Y \subseteq X$ closed subschemes. Let $$H_B : D^{\rm Nori}\to \Z\mod$$ be the representation given by $(X, Y, n)\leadsto H_B^n(X (\C), Y(\C);\Z)$  the relative singular cohomology group after base change to the complex numbers.
\begin{defn}[Nori, see also {\cite[\S 9]{HMS}}]\label{Def:Nori}
 The abelian category
$$\ECM_k := \catN{H_B}$$
is the category of effective cohomological Nori motives. There is a non-effective version that we shall denote $\NM_k$.
\end{defn}
\begin{rem}The diagram $D^\Nori$ above agrees with the diagram
$\mathrm{Pairs}^\eff$ of \cite[Definition~9.1.1]{HMS}.
In loc.~cit.~the abelian categories are denoted by $\MMN_\Nori^\eff(k)$ and $\MMN_\Nori(k)$, respectively. Non-effective motives are obtained either by localisation of the diagram or of the category with respect to the Lefschetz motive
$\unit(-1)=(\mathbb{G}_m,\{1\},1)$. This is somewhat premature at this point as it involves the tensor structure. We are going to concentrate on the effective case.
\end{rem}

\subsection*{Tensor product via graded $\tensor$-quivers}
Let $D^{\Nori,\tensor}$ be the same quiver with, in addition, the following structure of a graded $\tensor$-quiver in the sense of Definition~\ref{defn:graded}. The grading is given by
\[ (X,Y,n)\mapsto \bar{n} \in \Z/2\Z\]
For vertices $(X,Y,n), (X',Y,n')$ we put
$$(X, Y, n)\otimes (X', Y', n') := (X \times_k X', X \times_k Y' \cup Y\times_k X', n+n')$$
making use of the product in $\Sch_k$. We choose the vertex $\unit$ and the edges $\id,\alpha,\beta,\beta', u,u'$ in the canonical way, e.g., the unit $\unit = (\Spec(k), \emptyset , 0)$,
$$u : (X, Y, n) \to (\Spec(k), \emptyset , 0)\otimes (X, Y, n)$$ and $u' : ({\rm Spec} (k), \emptyset , 0)\otimes (X, Y, n)\to (X, Y, n)$
the canonical maps. As relations we use the relations required by Definition~\ref{defn:graded}. All this is completely parallel to \cite[\S 9.3]{HMS}. By construction we obtain a graded $\tensor$-quiver.

Recall that singular cohomology $H_B^*$ is provided with a natural cross or external product
$$\kappa^B_{n, n'}: H^n_B(X,Y)\tensor H^{n'}_B(X',Y')\to H^{n+n'}_B(X\times_k X',X\times_k Y'\cup Y\times_k X')$$
Note that the representation $H_B$ is \emph{not} a $\tensor$-representation since $\kappa^B_{n, n'}$ fails to be an isomorphism, in general.

Following Nori, we set $D^{\good,\tensor}$ for the full sub-$\tensor$-quiver of vertices $(X,Y,n)$ such
that $H^*_B(X,Y)$ is concentrated in degree $n$ and free as a $\Z$-module.\\

\begin{lemma}Betti cohomology $H_B^\good := H_B|_{D^\good}:D^{\good,\tensor}\to \Z\mod$ is
a graded $\tensor$-representation with values in the subcategory
$(\Z\mod)^\flat$ of free $\Z$-modules of finite type.
\end{lemma}
\begin{proof}On good pairs, the map $\kappa^B_{n,n'}$ is indeed an isomorphism by the K\"unneth formula.
The relations of the tensor quiver are all mapped to equalities in $\Z\mod$
by the standard properties of singular cohomology. Most are checked explicitly in \cite[Proposition~9.3.1]{HMS}. The remaining ones (e.g., concerning the
inverse $u'$ of $u$) are obvious.
\end{proof}
Indeed, our definition of a graded $\tensor$-quiver was modeled on this case.

\begin{cor}The abelian category $\catN{H_B^\good}$ carries a natural
$\tensor$-structure compatible with the forgetful functor to $\Z\mod$.
\end{cor}
\begin{proof}See Theorem~\ref{thm:Nori_graded}.
\end{proof}

Nori's Basic Lemma comes into play in comparing the universal categories for the two diagrams.
\begin{thm}[Nori, see {\cite[Theorem~9.2.22]{HMS}}]
The quiver $D^\Nori$ can be represented in $\catN{H_B^\good}$ in a compatible way with $H_B$. In particular,
\[ \ECM_k\isom \catN{H_B^\good}\]
carries a natural tensor structure.
\end{thm}
In the above, we are copying Nori's approach, but replace his approach to the
universal abelian category and its tensor product with the one developed in this paper.  In \cite{BVP2} we go further, providing a general axiomatic
framework for tensor motivic categories associated to a cohomological
functor on a suitable base category; the tensor structure is induced, using our main theorem,
by the cartesian tensor structure on the base category via a
cohomological K\"unneth formula.

We now turn to yet a different approach which does not mention $D^\Nori$ and $D^\good$ (at least not obviously so).

\subsection*{Tensor product via triangulated motives}
Let $\DMgm(k,\Q)$ be Voevodsky's category of geometric motives over $k$ with rational coefficients. Let
\[ R_B:\DMgm(k,\Q)\to D^b(\Qvsp)\]
be the Betti-realisation. It maps the motive of an algebraic variety to its singular cochain complex.

\begin{rem}The existence of the Betti-realisation is completely straightforward.
The first reference with rational coefficients is
 \cite{Hu1}, \cite{Hu2} as a byproduct of a  functor into mixed realisations.
With integral coefficients it is formulated in \cite{Harrer}.
In the original literature on motives, realisation functors were usually contravariant. This is also the viewpoint taken in the above references.

 More recently, Voevodsky and then Ayoub who, in \cite{Ayoub-Betti} constructs Betti-realisations for motives over any base, has been using the covariant point of view.

For our application, it does not matter which point of view is taken. We fix on the contravariant one because we want to refer to \cite{Harrer} later on.
\end{rem}

\begin{defn}Let $\MMN_k :=\catN{H^0_B}$ be the universal abelian category defined by the Betti-realisation.
\end{defn}

Based on a sketch of Nori, Harrer (see \cite{Harrer}) was able to show:

\begin{thm}[Harrer {\cite[Thm~7.3.1]{Harrer}}]\label{thm:real}
The Betti-realisation factors naturally via the bounded derived category of $\MMN_k$ and even that of
$\catNsub{0}{H^*_B}$.
\end{thm}
\begin{rem}
The proof is based on Nori's basic lemma: for every affine variety $X$ and subvariety $Y$, there is a subvariety $X\supset Z\supset Y$ such that the singular cohomology of the pair $(X,Z)$ is concentrated in the degree equal to the dimension of $X$, i.e., $(X,Z,\dim X)$ is a good pair. As pointed out by Nori, this can be used in order to construct, for every affine variety $X$, a natural
complex of motives. Using \v{C}ech-complexes, this extends to all varieties.
Harrer's main effort was to establish functoriality of the construction with respect to finite correspondences. When working with rational coefficients (as we do), functoriality with respect to morphisms is enough, see \cite{Ivorra2}, \cite{HMS}. Harrer's result is formulated for $\NM_k$, but actually proved for $\catN{H_B^\good}$.
The same proof also works
without change for $\MMN_k=\catN{H^0_B}$ and even  the refinement $\catNsub{0}{H^*_B}$.
\end{rem}

\begin{thm}The category $\MMN_k$ carries a natural tensor structure such that
$\MMN_k\to\Qvsp$ is a tensor functor and $\DMgm(k,\Q)\to D^b(\MMN_k)$ is a triangulated
tensor functor.

In particular, $\MMN_k$ is Tannakian.
\end{thm}
\begin{proof}We apply Corollary~\ref{cor:tensor_case} to the rigid tensor category $\DMgm(k,\Q)$ and the Betti-realisation. The assumption is satisfied
by Theorem~\ref{thm:real}.
 This makes $\MMN_k$ a rigid tensor category; the
Betti-realisation $\MMN_k\to\Qvsp$ is a fibre functor.
\end{proof}

\begin{defn}The \emph{motivic Galois group} of $k$ is defined
as the Tannakian dual of the $\MMN_k$.
\end{defn}

\begin{prop}$\MMN_k$ is naturally equivalent to Nori's original category, i.e.,  $\NM_k \cong \MMN_k$.
The motivic Galois group is naturally isomorphic to Nori's original motivic
Galois group.
\end{prop}
\begin{proof}
For the abelian category, this is already shown in \cite{HMS}. The tensor structures are based on the K\"unneth formula. In each case it is uniquely determined by its value for very good pairs, hence they are the same. The statement about the motivic Galois group follows.
\end{proof}

\begin{rem}The whole argument also works for motives with coefficients in any field (including finite fields) or Dedekind domain (in particular the integers). Corollary~\ref{cor:tensor_case_right} can be used instead of the more straightforward Corollary~\ref{cor:tensor_case}.
Harrer's work in \cite{Harrer} handles the integral case.
\end{rem}

\end{document}